\def\lmjrunauthor{K. Ford}%Example:% F. Author\sep S. Author \and Th. Author
\def\lmjruntitle{Hardy-Ramanujan inequality for sifted sets}%Example:% RunTitle of Article
\def\lmjmsc{11N25}
\def\lmjkeywords{Hardy-Ramanujan inequality, shifted primes, sieves.}
\documentclass{lmj}
\include{lmj.def}
\LMJCLS%
\usepackage{amsmath,amsfonts,amssymb,times,latexsym,mathrsfs}
\newcommand{\NN}{{\mathbb N}}

\newcommand{\cE}{\ensuremath{\mathcal{E}}}

\newcommand{\cQ}{\ensuremath{\mathcal{Q}}}
\newcommand{\cR}{\ensuremath{\mathcal{R}}}
\newcommand{\cS}{\ensuremath{\mathcal{S}}}

\newcommand{\sG}{\ensuremath{\mathscr{G}}}

\newcommand{\sR}{\ensuremath{\mathscr{R}}}

\newcommand{\sT}{\ensuremath{\mathscr{T}}}

\newcommand{\er}{\ensuremath{\mathrm{e}}}

\newcommand{\be}{\begin{equation}}
\newcommand{\ee}{\end{equation}}

\newcommand{\ssum}[1]{\sum_{\substack{#1}}}  %%% stacked sum
\newcommand{\sprod}[1]{\prod_{\substack{#1}}}  %% stacked product

\newcommand{\eps}{\ensuremath{\varepsilon}}

\renewcommand{\le}{\leqslant}

\renewcommand{\ge}{\geqslant}

\newcommand{\pfrac}[2]{\left(\frac{#1}{#2}\right)}  %%% frac with paren

\begin{document}
\bibliographystyle{plainlmj}
\begin{topmatter}
%************* add runauthor at beginning of Your file
%************* add runtitle at beginning of Your file
%************* add title at beginning of Your file
%************* add add all authors at beginning of Your file
%************* add LMJRuntitle and LMJRunauthor  at beginning of Your file

%%%%%%%%%%%%%%%%%%%%%%%%%%%%%%%%%%%%%%%%%%%%%%%%%%%%%%%%%%%%%%%%%%%%%%%
%
%  Main Document
%
%%%%%%%%%%%%%%%%%%%%%%%%%%%%%%%%%%%%%%%%%%%%%%%%%%%%%%%%%%%%%%%%%%%%%%%

  % this had to be redefined
%\def\dedicated#1{\vskip6pt\vbox{\let\thanks\@gobble\nohyphen\centering%
%\normalsize\em{#1}\expandafter\elem@nothanks#1\thanks\relax}\vskip1pt}

\title{A Hardy-Ramanujan type inequality for shifted primes and sifted sets}
\author{Kevin Ford}

\institution{Department of Mathematics, 1409 West Green Street, University
of Illinois at Urbana-Champaign, Urbana, IL 61801, USA}
\email{ford@math.uiuc.edu}

\begin{abstract}
We establish an analog of the Hardy-Ramanujan inequality for 
counting members of sifted sets with a given number of distinct prime factors.
In particular, we establish a bound for the
number of shifted primes $p+a$ below $x$ with $k$ distinct prime factors,
uniformly for all positive integers $k$.
\end{abstract}

\vskip 6pt

% \dedicated always puts "Dedicated to" at the front

\begin{center}
\textit{In memory of Jonas Kubilius on the 100th anniversery of his birth.}
\end{center}

\end{topmatter}
\LMJarticle % be komentaro! negali būti apatinės eilutės

%%%%%%%%%%%%%%%%%%%%%%%%%%%%%%%%%%%%%%%%%%%%%%%%%%%%%%%%%%%%%%%%%%%%%%%%%
%%%%%%%%%%%%%%%%%%%%%%%%%%%%%%%%%%%%%%%%%%%%%%%%%%%%%%%%%%%%%%%%%%%%%%%%%

\section{Introduction}

The distribution of the number, $\omega(n)$, of distinct prime factors of a positive
integer $n$ has been well studied during the past century.
In 1917, Hardy and Ramanuajan \cite{HaRa} proved the inequality
\be\label{HR}
\pi_k(x) := \ssum{n\le x \\ \omega(n)=k} 1 \le C_1 \frac{x}{\log x}\; \frac{(\log\log x+C_2)^{k-1}}{(k-1)!},
\ee
where $C_1,C_2$ are certain absolute constants.
An asymptotic $\pi_k(x)$, valid for each \emph{fixed} $k$,
had earlier been proved by Landau in 1900.
The chief importance of \eqref{HR} lies in the \emph{uniformity} in $k$,
and it is this feature which allowed
Hardy and Ramanujan deduced from \eqref{HR} that $\omega(n)$ has 
\emph{normal order}\footnote{For a set of integers $n$ with counting
function $x-o(x)$ as $x\to\infty$, we have $\omega(n)=(1+o(1))\log\log n$ as $n\to\infty$.}
$\log\log n$.
An asymptotic for $\pi_k(x)$, uniform for $k\le C_3 \log\log x$
and arbitrary fixed $C_3$, was proved by Sathe and Selberg in 1954.
Thanks to subsequent work of a number of authors, notably Hildebrand 
and Tenenbaum \cite{HT88}, a uniform asymptotic for $\pi_k(x)$ is known in a much wider
range $k\le c \frac{\log x}{(\log\log x)^2}$, $c>0$ some constant.
The right side of \eqref{HR} represents the correct order of magnitude
of $\pi_k(x)$ when $k=O(\log\log x)$, but is slightly too large when
$k/\log\log x \to \infty$ as $x\to\infty$.
See Ch. II.6 in \cite{Tenbook} for a more detailed history of the problem
and concrete formulas for $\pi_k(x)$.

The Hardy-Ramanujan inequality \eqref{HR} has been extended and generalized 
in many ways, such as replacing the summation over $n\le x$ with
a restricted sum over shifted primes
\cite{erdos35, Tim95}, replacing the summand 1 with a 
multiplicative function \cite{PP},
counting integers with a prescribed number
of prime factors in disjoint sets \cite[Theorem 3]{FPoisson},
counting the prime factors of polynomials at integer arguments
\cite{ten18}, counting integers with
$\omega(n)=k_1$ and $\omega(n+1)=k_2$ simultaneously
\cite[Th. 18]{G}
or replacing $\omega(n)$ with an arbitrary
additive function \cite{kubilius, elliott}.

In this note we establish an analog of the Hardy-Ramanujan
theorem, with complete uniformity in $k$, for
prime factors of integers restricted by a sieve condition.
The main theorem is rather technical and we defer the precise statement
to Section \ref{sec:mainthm}.
Here we describe some corollaries which are easier to digest.

\subsection{Notation conventions.}
Constants implied by the $O$- and $\ll$-symbols are independent
of any parameter except when noted by a subscript, e.g. $O_\eps()$ means
an implied constant that depends on $\eps$.
We denote $\pi(x)$ the number of primes which  are $\le x$.
The symbols $p$ and $q$ always denote primes.

\subsection{Application: prime factors of shifted primes.}

Let $a$ be a nonzero integer.  The distribution of the prime factors of 
numbers $p+a$, $p$ being prime, plays a central role in investigations of
Euler's totient function, the sum of divisors function, orders and
primitive roots modulo primes, and primality testing algorithms
(for these applications, $a=\pm 1$).
It is expected that the distribution of the prime
factors of a random shifted prime $p+a\le x$ behaves very much
like the distribution of the prime factors of a random integer in $[1,x]$.
A complicating factor is that
the distribution of the large prime factors of $p+a$, say those $>\sqrt{p}$,
is poorly understood.  For example, Baker and Harman \cite{BH}
showed that infinitely often, $p+a$ has a prime factor
at least $p^{0.677}$, and this is not known with
$0.677$ replaced by a larger number.

In 1935, Erd\H os \cite{erdos35} proved that
 the function $\omega(p-1)$ has normal order
$\log\log p$ over primes $p$.  To show this, Erd\H os proved
an upper bound of Hardy-Ramanujan type for the number
of primes $p\le x$ with $\omega(p-1)=k$ in a restricted range of $k$.  
The bounds were sharpened by Timofeev \cite{Tim95},
who proved a conjecturally best possible upper bound
when $k=O(\log\log x)$.
Here we extend this bound to hold uniformity for all $k$, uniformity in $a$,
and correct a small error in Timofeev's bound when $a$ is odd.

\begin{cor}\label{cor:shift}
Let $a\ne 0$ and define $s=2$ if $a$ is odd, and $s=1$ if $a$ is even.
Then uniformly for $k\in \NN$, $x\ge 2|a|$ and all $a\ne 0$ we have
\[
\# \{-a < p\le x: \omega\big( \tfrac{p+a}{s} \big) = k \} \ll
\frac{|a|}{\phi(|a|)} \pi(x) \frac{(\log\log x+O(1))^{k-1}}{(k-1)! \log x}.
\]
\end{cor}

We remark that Timofeev worked with $\omega(p+a)$ rather than
$\omega(\frac{p+a}{s})$; when $a$ is odd and $s=2$, dividing by $s$ is necessary 
because 2 always divides $p+a$ when $p\ge 3$.
The corresponding lower bound is not known for any $k$, 
although it is conjectured to hold for every $k$ satisfying $k=O(\log\log x)$.
The problem of the lower bound is intimately connected with the
\emph{parity problem} in sieve theory.  The best lower bound in
this direction is Theorem 3 of Timofeev \cite{Tim95}
which states (in the case $a=2$) that
\[
\# \{p\le x: \omega(p+2) \in \{k,k+1\} \} \gg
\pi(x) \frac{(\log\log x+O(1))^{k-1}}{(k-1)! \log x}
\]
uniformly for $1\le k\ll \log\log x$.  The case $k=1$ is a
the celebrate Theorem of J.-R. Chen.

\subsection{Application: integers with restricted factorization}

Let $\cE$ be any set of primes and
let $\cQ(\cE)$ be the set of positive integers, all of whose
prime factors belong to $\cE$.  Let
\be\label{Ex}
E(x) = \ssum{p\in \cE \\ p\le x} \frac{1}{p}.
\ee

The next corollary was established by Tenenbaum
\cite[Lemma 1]{ten00} using a different
method.

\begin{cor}\label{cor:set}
 Uniformly for all $\cE$ and all $k\in \NN$
we have
\[
\# \{ n\le x, n\in \cQ(\cE) : \omega(n)=k \} \ll x \frac{(E(x)+O(1))^{k-1}}{(k-1)! \log x}.
\]
\end{cor}

We also establish a count of shifted primes $p+a$ with 
a given number of prime factors, such that $p+a$ only has prime
factors from a given set, generalizing Corollaries \ref{cor:shift} and \ref{cor:set}.

\begin{cor}\label{cor:shift-set}
Let $a\ne 0$, and let $s=1$ if $a$ is even and $s=2$ is $a$ is odd.
Let $\cE$ be any nonempty set of primes, and define $E(x)$ by \eqref{Ex}.
 Uniformly for all $a$, all $x\ge 2|a|$, all $\cE$ and all $k\in \NN$
we have
\[
\# \{ -a<p\le x, \tfrac{p+a}{s}\in \cQ(\cE) : \omega\big(\tfrac{p+a}{s}\big)=k \} \ll \frac{|a|}{\phi(|a|)} \pi(x) \frac{(E(x)+O(1))^{k-1}}{(k-1)! \log x}.
\]
\end{cor}

\subsection{Application: the mean of twin primes}

Hardy and Littlewood conjectured in 1922 that
the number of prime $p\le x$ with $p+2$ also prime
is asymptotic to $C x/\log^2 x$ for some constant $C$.
At present, it is not known that there are infinitely
many such twin prime pairs.  Here we focus on the number of prime factors
of $p+1$ for such primes.

\begin{cor}\label{cor:twins}
Uniformly for $k\in \NN$ we have
\[
\# \{4<n\le x : n-1 \text{ and } n+1 \text{ are both prime}, \omega\big( \frac{n}{6} \big)=k\}
\ll x \frac{(\log\log x+O(1))^{k-1}}{(k-1)! \log^3 x}.
\]
\end{cor}

Again, we divide by 6 because all such $n$ are divisible by 6.

Corollaries \ref{cor:shift}, \ref{cor:set}, \ref{cor:shift-set} and \ref{cor:twins}
represent only a small sample of the type of bounds attainable 
using Theorem \ref{thm:main} below.  For example, we obtain
conjecturally best-possible (in the case $k=O(\log\log x)$)
upper bounds on the number of $n\le x$
with $\omega(n)=k$, and
with $n-1$ prime, $n+1$ prime, $n+5$ prime, and such that
$n$ has only prime factors from a given set.

As with \eqref{HR}, we expect the left sides in the
corollaries to be of smaller order than the right sides
 when $k/\log\log x \to \infty$ as $x\to \infty$.
 We will return to this in a subsequent paper.

%%%%%%%%%%%%%%%%%%%%%%%%%%%%%%%%%%%%%%%%%%%%
%
%  Main Theorem
%
%%%%%%%%%%%%%%%%%%%%%%%%%%%%%%%%%%%%%%%%%%%%%%
\section{Statement of the Main Theorem}\label{sec:mainthm}
%
%%%%%%%%%%%%%%%%%%%%%%%%%%%%%%%%%%%%%%%%%%%%%%

Here we state our main theorem and prove Corollaries \ref{cor:shift}, \ref{cor:set}, \ref{cor:shift-set} and \ref{cor:twins}.

Let $\sG(A)$ denote the set of non-negative multiplicative functions satisfying
\be\label{GA}
g(p^v) \le \frac{A}{p^v}
\qquad (p\text{ prime},v\in \NN).
\ee

An immediate consequence of \eqref{GA} and Mertens' theorems is
\be\label{Gx}
G(x) := \sum_{p\le x} g(p) \le A(\log\log x+O(1)), \qquad (x\ge 2).
\ee

\begin{thm}\label{thm:main}
Let $\cS$ be a set of positive integers, and let $s$ be any integer
dividing every element of $\cS$.
Suppose that $g\in \sG(A)$, $x\ge s^2$, $B\ge \exp\{-\sqrt{\log x}\}$ and $\lambda>0$ is a constant
so that
\be\label{sieve}
\# \bigg\{\text{prime } q \le \frac{x}{rs}  : qrs\in \cS \bigg\} \le B x \frac{g(r)}{\log^\lambda (\frac{2x}{rs})} \qquad (1\le r \le x/s).
\ee
Then, uniformly for positive integers $k$,
\[
\# \bigg\{ n \le x, n\in \cS : \omega(n/s)=k \bigg\} \ll_{\lambda,A}
B x \frac{(G(x) + O_A(1))^{k-1}}{(k-1)!\log^\lambda x}.
\]
\end{thm}

The proof of Theorem \ref{thm:main} will be given in the next section.
Here we discuss corollaries.

We first recover the original Hardy-Ramanujan inequality \eqref{HR}.
In this case $\cS=\NN$ and the left side of \eqref{sieve}
is $\pi(y/r)\ll (y/r)/\log(2y/r)$ by Chebyshev's estimates for primes.
Also, $g(r)=1/r$ for all $r$ and $g\in \sG(1)$.
Theorem \ref{thm:main} then implies \eqref{HR}.

\emph{Proof of Corollary \ref{cor:shift-set}}.
Let $\cS = \{ p+ a : p>-a \text{ prime}, \tfrac{p+a}{s}\in \cQ(\cE) \}$.  
Provided that $r\in \cQ(\cE)$, for all $y\ge rs$ we have by
a standard sieve bound (Corollary 2.4.1 in \cite{HR}) that
\begin{align*}
\# \bigg\{ q \le \frac{y}{rs}  :  qrs \in \cS \bigg\} &= \#
 \bigg\{ q \le \frac{y}{rs}  :  q, qrs-a\text{ both prime }\bigg\}\\
 &\ll \frac{|ars|}{\phi(|ars|)} \frac{y}{rs \log^2 (\frac{2y}{rs})}. 
\end{align*}
When $rs\not\in \cQ(\cE)$, the left side is zero.
Thus, defining $g(r)=1/\phi(r)$ when $r\in \cQ(\cE)$ and zero otherwise,
we see that \eqref{sieve} holds with $B \ll \frac{|a|}{\phi(|a|)}$.
Clearly $g \in \sG(2)$, and $G(x)=E(x)+O(1)$ since $g(p)=1/p+O(1/p^2)$
for $p\in \cE$.  Corollary \ref{cor:shift-set}
now follows from Theorem \ref{thm:main} when $x\ge 2|a|$ since
$x-a \asymp x$.
\qed

Corollary \ref{cor:shift} is a special case of 
Corollary \ref{cor:shift-set}, upon taking $\cE$ the set of all primes.

\medskip
\textbf{Remark.}  The author thanks Maciej Radziejewski for informing
him of a subtle issue, namely that
the set $g := \text{gcd} \{p+a : p+a\in \cQ(\cE) \}$ is not always
1 or 2.  For example if $\cE = \{3\}$ and $a=2$, then $g=9$.  When
$2|a$ and $\cE$ contains at least two odd primes, determination
of $g$ is a very difficult unsolved problem in general.  However,
we always have $s|g$, where $s$ is given in Corollary \ref{cor:shift-set}.
Thus, it is important in Theorem \ref{thm:main} that $s$ be \emph{any}
integer dividing every member of $S$.

\medskip

\emph{Proof of Corollary \ref{cor:set}}.
Let $\cS = \cQ(\cE)$.
Here we have $s=1$ (in particular, $1\in \cS$).
For any $r\le y$ we have
\[
\# \{ q\le y/r: qr\in \cS \} \le \begin{cases}
0 & \text{ if } r\not\in \cS \\
\pi(y/r) & \text{ if } r\in \cS.
\end{cases}
\]
By Chebyshev's bound for $\pi(x)$, 
 \eqref{sieve} holds with $\lambda=1$, $B=O(1)$ and
$g$ defined by $g(p^v)=1/p^v$ if $p\in \cE$, $g(p^v)=0$ if
$p\not\in \cE$.  Hence $g\in \sG(1)$ and $G(x)=E(x)$.
The Corollay follows from Theorem \ref{thm:main}.
\qed

\emph{Proof of Corollary \ref{cor:twins}}.
Let $\cS = \{n > 4: n-1, n+1\text{ both prime}\}$.
We have $s=6$.  By the sieve (e.g., Theorem 2.4 of \cite{HR}),
for any $r\le x/6$,
\begin{align*}
\# \{q\le \tfrac{x}{6r} : 6rq \in \cS \} \ll 
\frac{x g(r)}{\log^3\pfrac{x}{3r}}, \qquad
g(r) = \frac{1}{r} \sprod{p|r\\ p>3} \frac{1-1/p}{1-3/p}.
\end{align*}
Thus, \eqref{sieve} holds and $g\in \sG(2)$.   Since $g(p) = \frac{1}{p}+O\pfrac{1}{p^2}$,
 $G(x)=\log\log x + O(1)$,
and the corollary follows.
\qed

%%%%%%%%%%%%%%%%%%%%%%%%%%%%%%%%%%%%%%%%%%%%%%%%%%%
%
%
\section{Proof of Theorem \ref{thm:main}}

We begin with a technical lemma.  Here $P^+(r)$ is  the largest prime factor
of $r$, with $P^+(1):=0$.

\begin{lem}\label{omega-lem}
Let $\lambda\ge 0$ and $g\in \sG(A)$.  Uniformly for $x\ge 2$ and $\ell\ge 0$ we have
\[
\ssum{\omega(r)=\ell \\ r P^+(r)\le x} \frac{g(r)}{\log^\lambda (x/r)}
\ll_{A,\lambda} \frac{(G(x)+O_A(1))^{\ell}}{\ell!\log^\lambda x}.
\]
\end{lem}

\begin{proof}
If $\ell=0$ then the only summand corresponds to $r=1$ and the result is trivial.
Now suppose $\ell\ge 1$.  Then $2\le r\le x/2$.
We separately consider $r$ in special ranges.  Let 
$Q_j = x^{1/2^j}$ for $j\ge 0$ and define
\[
\sT_j = \bigg\{ r \in \bigg[2,\frac{x}{2} \bigg] \cap \bigg[\frac{x}{Q_{j-1}}, \frac{x}{Q_j}\bigg] : \omega(r)=\ell, rP^+(r)\le x \bigg\}, \quad j\ge 1.
\]
For $r\in \sT_j$, we have $P^+(r) \le x/r \le Q_{j-1}$.
Also,  if $\sT_j$ is nonempty then $Q_{j-1} \ge 2$ and $j\ge 1$.  We have
\begin{align*}
\sum_{r\in \sT_j} \frac{g(r)}{\log^\lambda (x/r)} &\le \frac{1}{\log^\lambda Q_j} \ssum{r\in \sT_j} g(r).
\end{align*}
For the sum on the right side,
we use the ``Rankin trick'' familiar from the study of smooth numbers.
 Let $\alpha = \frac{1}{20\log Q_j}$.
Since $Q_{j-1} \ge 2$, $Q_j \ge \sqrt{2}$ and thus $0<\alpha \le \frac16$.
From the definition \eqref{GA} of $\sG(A)$ we have 
\[
g(m) \le \frac{A^{\omega(m)}}{m} \ll_A m^{-1/2}.
\]
Hence, when $r\ge x/Q_{j-1}$,
\[
g(r) = g(r)^\alpha g(r)^{1-\alpha} \ll_A r^{-\alpha/2} g(r)^{1-\alpha}
\ll x^{-\alpha/2} g(r)^{1-\alpha},
\]
since $(x/r)^{\alpha/2} \le Q_{j-1}^{\alpha/2} =Q_j^{\alpha} = \er^{1/20}$.
Thus,
\be\label{lem-1}
\sum_{r\in \sT_j} \frac{g(r)}{\log^\lambda (x/r)} \ll \frac{x^{-\alpha/2}}{\log^\lambda Q_j} \sum_{r\in \sT_j} g(r)^{1-\alpha} \le 
\frac{x^{-\alpha/2}}{\log^\lambda Q_j} \ssum{P^+(r)\le Q_{j-1} \\ \omega(r)=\ell} g(r)^{1-\alpha}.
\ee
Using \eqref{GA} again,
\begin{align*}
\ssum{P^+(r)\le Q_{j-1} \\ \omega(r)=\ell} g(r)^{1-\alpha} &\le \frac{1}{\ell!}
\bigg\{ \sum_{p\le Q_{j-1}} g(p)^{1-\alpha}+g(p^2)^{1-\alpha}+\cdots\bigg\}^{\ell} \\
&= \frac{1}{\ell!}
\bigg\{ O_A(1)+ \sum_{p\le Q_{j-1}} g(p)^{1-\alpha} \bigg\}^{\ell}.
\end{align*}
If $g(p)\ge 1/p^2$ we have $g(p)^{-\alpha} \le p^{2\alpha}=1+O(\alpha \log p)$
when $p\le Q_{j-1}$.  Hence,
\begin{align*}
\sum_{p\le Q_{j-1}} g(p)^{1-\alpha} &\le \sum_{g(p) <1/p^2} g(p)^{5/6}
+ \ssum{p\le Q_{j-1} \\ g(p) \ge 1/p^2} g(p)(1+O(\alpha \log p)) \\
&\le O(1) + G(Q_{j-1}) + O_A (1),
\end{align*}
using \eqref{GA} again plus Mertens' theorems.
Thus,
\[
\ssum{P^+(r) \le Q_{j-1} \\ \omega(r)=\ell} g(r)^{1-\alpha} \le 
\frac{(G(x)+O_A(1))^\ell}{\ell!}.
\]
Inserting the last bound into \eqref{lem-1}, we see that for each $j$,
\begin{align*}
\sum_{r\in \sT_j} \frac{g(r)}{\log^\lambda (x/r)} 
&\le \frac{2^{\lambda j} \exp\big\{-\frac1{40}\cdot 2^j\big\}}{\log^\lambda x}\, \frac{(G(x) +O_A(1))^{\ell}}{\ell!}.
\end{align*}
Summing over $j$ completes the proof.
\end{proof}

%%%%%%%%%%%%%%%%%%%%%%%%%%%%%%%%%%%%%%%%%%%%%%%%%%%%%%%%%%%%%%%%
%%%%%%%%%%%%%%%%%%%%%%%%%%%%%%%%%%%%%%%%%%%%%%%%%%%%%%%%%%%%%%%%

\emph{Proof of Theorem \ref{thm:main}}.
Let $n\le x, n\in \cS$ and $\omega(n/s)=k$.
Define $q=P^+(n/s)$ and write $n=qrs$.  
If $q\nmid r$ then $\omega(r)=k-1$, and if $q|r$ then $\omega(r)=k$.
  Also, $rP^+(r) \le rq = n/s \le x/s$.
It follows that $r\in \sR_{k-1} \cup \sR_{k}$, where
\[
\sR_\ell = \{ r\in \NN : r P^+(r)\le x/s, \omega(r)=\ell \}.
\]
Using \eqref{sieve}, followed by Lemma \ref{omega-lem}, we have
for $\ell \in \{k-1,k\}$ the bounds
\be\label{omegarl}
\begin{split}
\# \{n\le x, n\in \cS : \omega(r)=\ell \} &\le 
\sum_{r\in \cR_{\ell}} \# \{ q\le x/(rs) : qrs \in \cS \} \\
&\le B x \sum_{r\in \cR_{\ell}} \frac{g(r)}{\log^\lambda (\frac{2x}{rs})}\\
&\ll_{\lambda,A} Bx \frac{(G(x)+O_A(1))^{\ell}}{\ell! \log ^\lambda (x/s)}\\
&\ll_{\lambda,A} Bx \frac{(G(x)+O_A(1))^{\ell}}{\ell! \log ^\lambda x},
\end{split}
\ee
using that $x\ge s^2$ in the last step.

When $k>\log\log x$ we use the crude estimate
$G(x) \ll_A \log\log x$ from \eqref{Gx} and deduce from \eqref{omegarl} that
\begin{align*}
\# \{n\le x, n\in \cS: \omega(n)=k \} &\ll_{\lambda,A}
Bx \frac{(G(x)+O_A(1))^{k-1}}{(k-1)! \log ^\lambda x}\bigg(1 + 
\frac{G(x)+O_A(1)}{k} \bigg) \\  &\ll_{\lambda,A} Bx \frac{(G(x)+O_A(1))^{k-1}}{(k-1)! \log ^\lambda x}.
\end{align*}

If $1\le k\le \log\log x$, we keep the $\ell=k-1$ term
from \eqref{omegarl}
and bound the $\omega(r)=k$ term in a different way.
If $\omega(r)=k$ then $q^2|(n/s)$.  Thus, using
 a crude version of the main theorem in \cite{IP}, and the
 hypothesized bound $B\ge  \exp\{-\sqrt{\log x}\}$, we deduce that
\begin{align*}
\# \{n\le x, n\in \cS : \omega(r)=k \}  &\le \# \{m \le x/s : P^+(m)^2|m \}\\
&\ll x \exp \{- 10\sqrt{\log x} \} \ll_{A,\lambda} B x \frac{(G(x)+O_A(1))^{k-1}}{(k-1)! (\log^\lambda x)}.
\end{align*}
The proof is complete.
\qed

\textbf{Acknowledgements}.  The author thanks R\'egis de la Bret\`eche
for drawing his attention to \cite{G} and Lemma 1 of \cite{ten00}.

%%%%%%%%%%%%%%%%%%%%%%%%%

\end{document}